\documentclass[10pt]{amsart}
\usepackage[utf8]{inputenc}

\usepackage{amssymb,amsthm,amsmath}
\usepackage{mathrsfs}
\usepackage{enumerate}
\usepackage{graphicx,xcolor}
\usepackage{setspace}
\usepackage[hidelinks]{hyperref}

\newcommand{\dd}{\mathrm{d}}
\newcommand{\E}{\mathbb{E}}

\newcommand{\1}{\textbf{1}}
\newcommand{\R}{\mathbb{R}}

\newcommand{\sC}{\mathscr{C}}

\newcommand{\scal}[2]{\left\langle #1, #2 \right\rangle}
\newcommand{\red}{}

\DeclareMathOperator{\Var}{Var}

\DeclareMathOperator{\vol}{vol}

\usepackage{xpatch}
\makeatletter
\def\thm@space@setup{%
  \thm@preskip=12pt plus 0pt minus 0pt
  \thm@postskip=0pt plus 0pt minus 0pt
}
\xpatchcmd{\proof}{6\p@\@plus6\p@\relax}{\z@skip}{}{}
\makeatother

\newtheorem{theorem}{Theorem}
\newtheorem{lemma}[theorem]{Lemma}

\theoremstyle{remark}
\newtheorem{remark}[theorem]{Remark}

\theoremstyle{definition}

\title{Stability of simplex slicing}

\author{Sergii Myroshnychenko}

\address{(SM) University of the Fraser Valley, Abbotsford, BC, Canada.\vspace*{-1em}}

\author{Colin Tang}

\author{Kateryna Tatarko}

\address{(KT) University of Waterloo, Waterloo, ON, Canada.}

\author{Tomasz Tkocz}

\address{(CT \& TT) Carnegie Mellon University; Pittsburgh, PA 15213, USA.}

\email{ttkocz@math.cmu.edu}

\thanks{SM was supported in part by NSERC Grant no 2024-05044. KT was supported in part by NSERC Grant no 2022-02961. CT and TT's research supported in part by NSF grant DMS-2246484.}

\date{\today}

\begin{document}

\begin{abstract} 
We establish dimension-free stability of Webb's sharp simplex slicing (1996). Incidentally, {\red we show the Lipschitzness of volume of central sections of arbitrary (not necessarily symmetric) convex bodies}. 
\end{abstract}

\maketitle

\bigskip

\begin{footnotesize}

\end{footnotesize}

\bigskip

\section{Introduction}

When a certain nontrivial inequality becomes known along with its precise equality conditions, a natural important next step often seems to be to investigate its stability properties, or quantitative versions. Bonnesen's inequality for the planar isoperimetric problem provides a good concrete example, see \cite{BZ, Gr}. In analytic and geometric contexts, this paradigm has been particularly fruitful in establishing new results as well as, perhaps even more significantly, developing new methods, see for instance {\red \cite{BBJ, CFP, FHT1, FHT2, FMP}} for a handful of recent results. 

This paper concerns stability estimates for volume of central hyperplane sections of the regular simplex. Such ``slicing'' inequalities have had a long and rich history, going back to Anthon Good's question from the 70s, who asked to determine subspaces of a given dimension that minimise the volume of sections of a centred unit cube. Besides its inherent geometric interest, this question had in fact origins and applications in Diophantine approximations and geometry of numbers, see, e.g. \cite{BV}. Important results include sharp bounds, in the case of symmetric convex bodies, for sections of $\ell_p$ balls (see \cite{Ball-cube, Bar, ENT, Hadw, Hen, Kol, MP, Vaa}), and in the case of nonsymmetric ones, only for the regular simplex (see \cite{W-phd, W}). For a comprehensive overview, we refer to the monograph \cite{Kol-mon} and the survey \cite{NT}. Despite these being decades old, only recently have quantitative results begun to develop, and only in the symmetric cases, only for hyperplane sections (see \cite{CNT, GTW}), with an interesting application in information theory (see \cite{MR}). 

This work continues the program and makes a first step in the nonsymmetric case of maximal volume sections of a simplex, which is far less understood than the symmetric case (needless to say, even the sharp lower bounds have not been available only until now, see the second author's recent work \cite{Tang}).

\section{Results}

We work in Euclidean spaces $\R^n$ equipped with the standard inner product $\scal{x}{y} = \sum_{j=1}^n x_jy_j$, the norm $|x| = \sqrt{\scal{x}{x}}$ and the standard basis $e_1, \dots, e_n$ ($e_j$ has $1$ at $j$-th coordinate and $0$ elsewhere). We put 
\[
\1_n = e_1 + \dots + e_n = (1,\dots,1)
\]
to be the vector with all $1$s. 

\subsection{Simplex slicing} 
We let
\[
\Delta_n = \text{conv}\{e_1, \dots, e_{n+1}\} = \left\{x \in \R^{n+1}, \ \sum_{j=1}^{n+1} x_j = 1, \ x_1, \dots, x_{n+1} \geq 0 \right\}
\]
be the $n$-dimensional regular simplex embedded in the hyperplane $\frac{1}{n+1}\1_{n+1} + \1_{n+1}^\perp$ in $\R^{n+1}$. We are interested in the volume of \emph{central} hyperplane sections of $\Delta_n$ of codimension $1$, that is $(n-1)$-dimensional sections passing through the barycentre $\frac{1}{n+1}\1_{n+1}$ of $\Delta_n$. Every such section is conveniently represented as $\Delta_n \cap a^\perp$, as $a$ ranges over all unit vectors in $\R^{n+1}$ orthogonal to $\1_{n+1}$, so with $\sum a_j = 0$. Using homogeneity of volume, we have the following probabilistic representation for the volume of slices,
\begin{equation}\label{eq:vol-formula}
\vol_{n-1}(\Delta_n \cap a^\perp) = \frac{\sqrt{n+1}}{(n-1)!}p_a(0), \qquad |a| = 1, \sum a_j = 0.
\end{equation}
Here and throughout,
\[
p_a \text{ is the probability density of } \sum_{j=1}^{n+1} a_jX_j,
\]
where $X_1, \dots, X_{n+1}$ are i.i.d. standard exponential random variables (with density $e^{-x}\1_{(0,\infty)}(x)$ on $\R$). For derivation, see Webb's paper \cite{W}, where he proved the following sharp upper bound on the volume of central slices
\begin{equation}\label{eq:Webb-main}
p_a(0) \leq \frac{1}{\sqrt{2}}
\end{equation}
with equality if and only if $a = \frac{e_j-e_k}{\sqrt{2}}$ for some $j \neq k$ (corresponding to the hyperplanes containing all but two vertices of $\Delta_n$). We obtain a refinement with optimal behaviour near the extremisers.

\begin{theorem}\label{thm:stab-sec}
Let $a = (a_1, \dots, a_{n+1})$ be a unit vector in $\R^{n+1}$ with $\sum a_j = 0$ and $a_1 \geq a_2 \geq \dots \geq a_{n+1}$. Set 
\begin{equation}\label{eq:def-delta}
\delta(a) = \left|a - \frac{e_1-e_{n+1}}{\sqrt{2}}\right|^2.
\end{equation}
Then
\begin{equation}\label{eq:stab-main}
p_a(0) \leq \frac{1}{\sqrt{2}} - \begin{cases} \frac{1}{10}\sqrt{\delta(a)}, & \delta(a) \leq \frac{1}{2000}, \\ 2\sqrt{2}\cdot 10^{-5}, & \delta(a) > \frac{1}{2000}. \end{cases}
\end{equation}
\end{theorem}
In particular, since $\sqrt{\delta(a)} \leq \sqrt{2}$, we have $p_a(0) \leq \frac{1}{\sqrt2} - 2\cdot 10^{-5}\sqrt{\delta(a)}$, for every unit vector $a$ as in Theorem \ref{thm:stab-sec}.

As for the stability of cube slicing (see (5) in \cite{CNT}), the deficit term is of the linear order in the Euclidean distance $\sqrt{\delta(a)}$ to the extremiser. The proof however is quite different for a fundamental reason: whilst for the cube there is a curious probabilistic formula for sections (in terms of negative moments, see e.g. (1) in \cite{CKT}) driving the main bootstrapping/self-improving argument, for the simplex, we have not been able to find a good analogue of such a formula, perhaps for the lack of symmetry (one may write a negative moment formula using Lemma 39 from \cite{NT}, but it is not clear how to ``evaluate'' or efficiently work with the limit). As a result, we have had to settle with working with \eqref{eq:vol-formula}. Our proof combines several \emph{ad hoc} arguments leveraging the exponential distribution as well as a solution to one rather standard optimisation problem over all probability densities on $\R$ with $L_\infty$ norm constrained. Webb gave in fact two completely different proofs of his upper bound \eqref{eq:Webb-main}: one Fourier-analytic and one relying solely on log-concavity. We crucially need to incorporate improvements coming from \emph{both} approaches, depending on {\red the} \emph{shape} of the outer-normal vector $a$.

We can check that the deficit term $\Theta(\sqrt{\delta(a)})$ is optimal by looking at the case $n=2$: if for instance we consider for small $\epsilon > 0$, sections by $a(\epsilon) = \sqrt{1-\epsilon}\frac{e_1-e_2}{\sqrt2} + \sqrt{\epsilon}\frac{e_1+e_2-2e_3}{\sqrt{6}}$, we get $\delta(a(\epsilon)) = \Theta(\epsilon)$ {\red (by a direct calculation)} and $p_{a(\epsilon)}(0) =  1 - \Theta(\sqrt{\epsilon})$ {\red (using \eqref{eq:vol-formula}, by an elementary geometric calculation of the segment length arising by intersecting the equilateral triangle by an appropriate line)}.

Within our method, we do not strive to derive best possible numerical values of the involved constants, giving priority to clarity of our arguments.

\subsection{Lipschitzness}
In the proof of the stability of cube slicing from \cite{CNT}, it was useful to know that the section function $a \mapsto \vol_{n-1}([-\tfrac12, \tfrac12]^n \cap a^\perp)$ is Lipschitz on the unit sphere $S^{n-1}$ in $\R^n$ (as a consequence of Busemann's theorem that requires the symmetry of the set). For a symmetric convex body $K$, the estimate on the Lipschitz constant of the function $E \mapsto \vol_{n-\ell} (K \cap E^\perp)$ where $E$ is an $\ell$-dimensional linear subspace is also known (see, for example, Proposition 2.3 in~\cite{ABB}). Investigating the stability of simplex slicing, we got naturally curious whether such property continues to hold for the simplex, or more generally for all convex bodies. It turns out that the answer is affirmative and is well-known to the experts in the field. As we have not found any specific reference for the {\red nonsymmetric} case, we include a~straightforward argument which relies on a functional version of Busemann's theorem (which in fact holds without symmetry/evenness assumptions) and application of {\red recent extensions \cite{MNRY, MSZ} of Gr\"unbaum's classical result from \cite{Gru}}.

We recall that a convex body in $\R^n$ is a compact convex set with nonempty interior. A convex body $K$ in $\R^n$ is called \emph{centred} if its \emph{barycentre}, $\text{bar}(K) = \frac{1}{\vol_n(K)}\int_K x \dd x$ is at the origin. The convex body $K$ is called \emph{isotropic} if it has volume $1$, is centred and its moment of inertia matrix is proportional to the identity,
\[
\left[\int_K x_ix_j \dd x\right]_{i,j \leq n} = L_K^2I_{n \times n},
\]
for a positive constant $L_K$ referred to as the \emph{isotropic constant} of $K$. 

A nonnegative function $f$ on $\R^\ell$ is called log-concave if $f = e^{-V}$ for a convex function $V\colon \R^\ell \to (-\infty, +\infty]$. It is a probability density if $\int_{\R^\ell} f = 1$ and then it is called \emph{centred} if $\int_{\R^\ell} xf(x) \dd x = 0$. The log-concave probability density $f$ is called \emph{isotropic} if it is centred and  its covariance matrix is the identity, $\int_{\R^\ell} x_ix_j f(x) \dd x = \delta_{ij}$. Then its \emph{isotropic constant} is defined to be $\|f\|_\infty^{1/\ell}$. The central problem in asymptotic convex geometry has been to estimate the following quantity
\begin{equation}\label{eq:Cl}
\sC_\ell = \sup \|f\|_\infty^{1/\ell},
\end{equation}
where the supremum is taken over all log-concave isotropic probability density functions $f$ on $\R^\ell$.
For comprehensive background, consult for instance Chapter 2 in~\cite{BGVV}. {\red The very recent exciting breakthrough results of Guan, as well as Klartag and Lehec, \cite{Gu, KL}, culminate in the assertion that $\sC_\ell = O(1)$.}

For two subspaces $E, F$ in $\R^n$ of the same dimension, we use the distance 
\[
d(E, F) = \|P_E - P_F\|_{HS},
\]
given by the Hilbert-Schmidt norm $\|P_E - P_F\|_{HS}$ of the difference between the orthogonal projection operators $P_E$ and $P_F$ onto $E$ and $F$ respectively.

\begin{theorem}\label{thm:sec-codim1}
Let $K$ be an isotropic convex body in $\R^n$. Let $1 \leq \ell < n$. For every $\ell$-dimensional subspaces $E, F$ in $\R^n$, we have
\[
\big| \vol_{n-\ell}(K \cap E^\perp) - \vol_{n-\ell}(K \cap F^\perp) \big| \leq e^{5\ell}\frac{\sC_\ell^{2\ell}}{L_K^\ell}d(E, F).
\]
\end{theorem}

\begin{remark}
It is well-known that for every isotropic convex body $K$ and the Euclidean ball $B$ of unit volume, $L_K \geq L_B \geq \frac{1}{12}$ (see for instance Lemma 6 in \cite{Ball}). {\red Together with $\sC_\ell = O(1)$, we get that the Lipschitz constant in Theorem \ref{thm:sec-codim1} is upper bounded by $C^\ell$ with a universal constant $C > 0$. This however seems suboptimal when $\ell = \omega(1)$ (in the extreme case, for line sections, i.e. when $\ell = n-1$, the Lipschitz constant can be bounded by $O(n^2)$, using standard arguments). It is perhaps an interesting problem to determine the optimal dependence.}
\end{remark}

\section{Proof of simplex slicing: Theorem \ref{thm:stab-sec}}

\subsection{Auxiliary results}

We begin with recording several elementary observations. By a standard exponential random variable, we mean a real-valued random variable with density $e^{-x}\1_{(0,\infty)}(x)$ on $\R$.

\begin{lemma}\label{lm:density2}
Suppose $X, Y$ are i.i.d. standard exponential random variables {\red and let} $a, b > 0$. Then $aX - bY$ has density
\begin{equation}\label{eq:g_ab}
g_{a,b}(x) = \frac{1}{a+b}\left(e^{-x/a}\1_{x \geq 0} + e^{x/b}\1_{x<0}\right).
\end{equation}
\end{lemma}
\begin{proof}
The densities of $aX$ and $-bY$ are $\frac{1}{a}e^{-x/a}\1_{x>0}$ and $\frac{1}{b}e^{x/b}\1_{x<0}$, respectively. By independence, the density is given by convolution, so for $x \geq 0$, we have
\[
g_{a,b}(x) = \frac{1}{ab}\int_x^\infty e^{-t/a}e^{(x-t)/b} \dd t = \frac{1}{ab}e^{x/b}\frac{1}{1/a+1/b}e^{-(1/a+1/b)x} = \frac{1}{a+b}e^{-x/a}.
\]
For $x < 0$, we proceed similarly.
\end{proof}

\begin{lemma}\label{lm:density2max}
Suppose $X, Y$ are i.i.d. standard exponential random variables {\red and let} $a, b > 0$.  Then $aX + bY$ has density on $\R$ uniformly upper bounded by $\frac{1}{e}\frac{1}{\min\{a,b\}}$.
\end{lemma}
\begin{proof}
Without loss of generality, we assume that $a < b$. The random variable $aX + bY$ is positive a.s. and at $x > 0$ has density
\[
\int_0^x \frac{1}{ab}e^{-t/a-(x-t)/b}\dd t = \frac{1}{ab}e^{-x/b}\int_0^x e^{-t(1/a-1/b)} \dd t \leq \frac{1}{ab}e^{-x/b}\int_0^x \dd t = \frac{1}{a}\frac{x}{b}e^{-x/b}.
\]
The standard inequality $ye^{-y} \leq e^{-1}$, $y \geq 0$, gives the result.
\end{proof}

We shall need a solution to a maximisation problem involving integration against densities $g_{a,b}$.

\begin{lemma}\label{lm:max-g_ab}
Let $a, b, C > 0$. Let $g_{a,b}$ be the probability density defined in \eqref{eq:g_ab}. For every probability density $f$ on $\R$ with $\|f\|_\infty \leq C$, we have
\[
\int_{\R} f\cdot g_{a,b} \leq C\left(1 - \exp\left\{-\frac{1}{C(a+b)}\right\}\right).
\]
\end{lemma}
\begin{proof}
Let $\alpha, \beta > 0$ be such that $C(\alpha + \beta) = 1$ and $\frac{\alpha}{a} = \frac{\beta}{b}$, the second condition being equivalent to $g_{a,b}(\alpha) = g_{a,b}(-\beta)$. Then $\alpha, \beta$ are uniquely determined in terms of $a,b, C$, namely
\[
\alpha = \frac{1}{C}\frac{a}{a+b}, \qquad \beta = \frac{1}{C}\frac{b}{a+b}.
\]
We consider the uniform density,
\[
f^*(x) = C\1_{[-\beta, \alpha]}(x).
\]
We shall argue that in order to maximise $\int fg_{a,b}$ subject to $\|f\|_\infty \leq C$, it is beneficial to move mass so that $f$ becomes $f^*$. Denote $m = g_{a,b}(\alpha) = g_{a,b}(-\beta)$ and observe that by the monotonicity of $g_{a,b}$, we have
\[
\min_{[-\beta,\alpha]} g_{a,b} = m = \max_{\R \setminus [-\beta,\alpha]} g_{a,b}.
\]
We break the integral into two bits,
\[
\int_{\R} (f^*-f)g_{a,b} = \int_{[-\beta,\alpha]} (f^*-f)g_{a,b}  + \int_{\R\setminus [-\beta,\alpha]}(f^*-f)g_{a,b}.
\]
Note that on $[-\beta,\alpha]$, we have $f^*-f  = C-f \geq 0$ and $g_{a,b} \geq m$, whilst on $\R\setminus [-\beta,\alpha]$, we have $f^*-f = -f \leq 0$ and $g_{a,b} \leq m$. Thus,
\[
\int_{\R} (f^*-f)g_{a,b} \geq \int_{[-\beta,\alpha]} (f^*-f)m  + \int_{\R\setminus [-\beta,\alpha]}(f^*-f)m = m \int_{\R} (f^*-f) = 0.
\]
As a result,
\[
\int_{\R} f\cdot g_{a,b} \leq \int_{\R} f^*\cdot g_{a,b} = C\int_{[-\beta,\alpha]}g_{a,b} = C\left(1 - \exp\left\{-\frac{1}{C(a+b)}\right\}\right),
\]
which finishes the proof.
\end{proof}

To employ the previous lemma, we shall use the following sharp bound on $L_\infty$ norms of log-concave densities in terms of their variance. 

\begin{lemma}[Bobkov-Chistyakov \cite{BCh}, Fradelizi \cite{Fr}]\label{lm:Bobkov}
For a log-concave probability density $f$ on $\R$, we have
\[
\|f\|_\infty \leq \left(\int x^2f(x) \dd x - \left(\int x f(x) \dd x\right)^2\right)^{-1/2}.
\]
\end{lemma}

Note that this is sharp: the equality holds for a (one-sided) exponential density.

We record for future use the following Fourier-analytic bound (with the main idea going back to \cite{Ball-cube, Haa, Hen}).

\begin{lemma}\label{lm:Fourier}
For $0 < x < 1$, define
\[
\Psi(x) = \frac{\sqrt{2}}{\pi}\int_0^\infty (1+x t^2)^{-\frac{1}{2x}} \dd t.
\]
Then, with the notation of Theorem \ref{thm:stab-sec},
\[
p_a(0) \leq \frac{1}{\sqrt{2}}\prod_{j=1}^{n+1} \Psi(a_j^2)^{a_j^2}.
\]
Moreover, $\Psi(x)$ is strictly increasing, $\Psi(\frac12) = 1$ and
\[
\Psi(x) = \frac{1}{\sqrt{2\pi x}}\frac{\Gamma\left(\frac{1}{2x}-\frac{1}{2}\right)}{\Gamma\left(\frac{1}{2x}\right)}.
\]
\end{lemma}
\begin{proof}
The inequality follows from the Fourier inversion formula, 
\[
p_a(0) = \frac{1}{2\pi}\int_{\R} \phi_{\sum a_jX_j}(t) \dd t,
\]
where $\phi_{\sum a_jX_j}(t) = \E e^{it\sum a_jX_j} = \prod \frac{1}{1+ia_jt}$ is the characteristic function, and H\"older's inequality. For details, see Webb's derivation from \cite{W}, pp.23-24. To get the explicit formula for $\Psi$ in terms of the $\Gamma$-function, we use a standard change of variables $s = (1+xt^2)^{-1}$ to reduce the integral to the $B$-function. To see about monotonicity, observe that for every fixed $t>0$, the integrand $(1+x t^2)^{-\frac{1}{2x}}$ is strictly increasing as a function of $x$: for instance, we check that $(-\frac{1}{x}\log(1+xt^2))' = \frac{1}{x^2}(\log(1+xt^2) - \frac{xt^2}{1+xt^2}) > 0$.
\end{proof}

Finally, we shall also need Webb's bound relying on {\red the log-concavity of exponentials}.

\begin{lemma}[Webb, \cite{W}, p.25]\label{lm:a1}
With the notation of Theorem \ref{thm:stab-sec},
\[
p_a(0) \leq \frac{1}{2\max\{a_1, {\red -a_{n+1}}\}}.
\]
\end{lemma}

\subsection{Proof of Theorem \ref{thm:stab-sec}}

Throughout, we assume that $a$ is a unit vector in $\R^{n+1}$ with $\sum a_j = 0$ and, {\red without loss of generality,}
$a_1 \geq \dots \geq a_l > 0 > a_{l+1} \geq \dots \geq a_{n+1}$.

We let
\[
u = a_1, \qquad v = -a_{n+1},
\]
as they play a prominent role in the ensuing analysis. Recall definition \eqref{eq:def-delta},
\[
\delta = \delta(a) = \left|a - \frac{e_1-e_{n+1}}{\sqrt{2}}\right|^2 = \left|u - \frac{1}{\sqrt{2}}\right|^2 +  \left|v - \frac{1}{\sqrt{2}}\right|^2 + \sum_{j=2}^n a_j^2 = 2 - \sqrt{2}(u+v).
\]
When $n=1$, $u = v = \frac{1}{\sqrt{2}}$ and $\delta(a) = 0$, we plainly have equality in \eqref{eq:stab-main} and there is nothing to prove. We shall now assume that $n \geq 2$ and break the proof into two main cases, depending on whether $\delta(a)$ is small or bounded away from $0$ by a constant.

\subsubsection{Stability near the maximiser: $\delta(a) \leq \frac{1}{2000}$}

Our goal in this part is to show 
\[
p_a(0) \leq \frac{1}{\sqrt{2}} - \frac{1}{10}\sqrt{\delta(a)}.
\]
With hindsight, we define the parameter 
\[
\sigma = \sqrt{1-u^2-v^2}
\]
and we need two different arguments depending on whether $\sigma$ is large relative to $\sqrt{\delta(a)}$ or not. {\red For brevity, we shall write $\delta = \delta(a)$.}

\emph{Case 1: $\sigma \geq \frac{1}{\sqrt2}\sqrt{\delta}$.}
We split the sum $S = \sum_{j=1}^{n+1} a_jX_j$ up into two pieces,
\[
S = X - Y, \qquad X = uX_1 - vX_{n+1}, \qquad Y = -\sum_{j=2}^{n} a_jX_j
\]
and let $f_X$, $f_Y$ be the densities of $X$ and $Y$, respectively. By the independence of $X$ and $Y$, $p_a$, the density of $S$, is the convolution of $f_X(\cdot)$ and $f_Y(-\cdot)$. In particular,
\[
p_a(0) = \int_{\R} f_X \cdot f_Y.
\]
Note that $\Var(Y) = \sum_{j=2}^n a_j^2 = 1- u^2-v^2 = \sigma^2$. Consequently, by Lemma \ref{lm:Bobkov} ($Y$ is log-concave as a sum of independent log-concave random variables), $\|f_Y\|_\infty \leq \frac{1}{\sigma}$.
By Lemma \ref{lm:density2}, $f_X = g_{u,v}$, so Lemma \ref{lm:max-g_ab} with $C = \frac{1}{\sigma}$ gives
\[
p_a(0) \leq \frac{1}{\sigma}\left(1 - \exp\left\{-\frac{\sigma}{u+v}\right\}\right).
\]
By the monotonicity of slopes of the concave function 
\[
x \mapsto 1-e^{-x},
\]
the function 
\[
x \mapsto \frac{1-e^{-x}}{x}
\]
is decreasing on $(0,\infty)$ {\red and} so is the right hand side above as a function of $\sigma$. Thus, using the assumed bound $\sigma \geq \sqrt{\delta/2}$, we get
\[
p_a(0) \leq \frac{1}{\sqrt{\delta/2}}\left(1 - \exp\left\{-\frac{\sqrt{\delta/2}}{u+v}\right\}\right).
\]
Recall $u+v = \frac{2-\delta}{\sqrt{2}}$. Note that for $0 \leq y \leq 1$, we clearly have 
{\red \[
1 - e^{-y} \leq y - \frac{y^2}{2}+\frac{y^3}{6} \leq y - \frac{y^2}{3},
\]}
which applied to $y = \frac{\sqrt{\delta/2}}{u+v} = \frac{\sqrt{\delta/2}}{\sqrt2 - \delta/\sqrt{2}}  < 1$ {\red (since $\delta \leq \frac{1}{2000}$)} yields
\[
p_a(0) \leq \frac{1}{u+v} - \frac{1}{3}\frac{\sqrt{\delta/2}}{(u+v)^2}.
\]
Using in the first term, 
{\red \[
\frac{1}{u+v} = \frac{1}{\sqrt{2}-\delta/\sqrt{2}} < \frac{1}{\sqrt2} + \frac{\delta}{\sqrt{2}}
\]}
(the left hand side is convex in $\delta$, the right hand side is linear and they agree at $\delta = 0$ and $\delta = 1$), whereas in the second term, crudely, $u+v < \sqrt2$, we arrive at
\[
p_a(0) \leq \frac{1}{\sqrt2} + \frac{\delta}{\sqrt{2}} - \frac{\sqrt\delta}{6\sqrt2} \leq \frac{1}{\sqrt2} - \frac{1}{10}\sqrt\delta,
\]
as long as $\delta < \left(\frac16 - \frac{\sqrt2}{10}\right)^2 = 0.00063..$. This concludes the argument for Case 1.

\emph{Case 2: $\sigma < \frac{1}{\sqrt2}\sqrt{\delta}$.} Without loss of generality, we can assume that $u \geq v$ (by changing $a$ to $-a$, if necessary, which leaves $p_a(0)$ unchanged). By Lemma \ref{lm:a1},
\[
p_a(0) \leq \frac{1}{2u}.
\]
We shall argue that this is at most the desired $\frac{1}{\sqrt2} - \frac{1}{10}\sqrt{\delta}$. It is convenient to work with new rotated (nonnegative) variables $s = \frac{u+v}{\sqrt{2}}$ and $t = \frac{u-v}{\sqrt{2}}$, so that $s^2 + t^2 = u^2+v^2$. Then, the imposed constraint $\sigma < \frac{1}{\sqrt2}\sqrt{\delta}$ becomes
\[
1-s^2-t^2 < \frac12\delta = 1-s.
\]
This gives the lower bound $t > \sqrt{s(1-s)}$. Consequently,
\[
p_a(0) \leq \frac{1}{2u} = \frac{1}{\sqrt{2}(s+t)} < \frac{1}{\sqrt2(s+\sqrt{s(1-s)})}.
\]
Finally, observe that
\[
s+\sqrt{s(1-s)} = 1 - \tfrac{\delta}{2} + \sqrt{\tfrac{\delta}{2}\left(1- \tfrac{\delta}{2}\right)} > 1 + \tfrac{\sqrt{\delta}}{2},
\]
as $\delta \leq \frac{1}{2000}$, so in this range,
\[
p_a(0) < \frac{1}{\sqrt{2}(1 + \sqrt{\delta}/2)} \leq \frac{1}{\sqrt{2}} - \frac{\sqrt{\delta}}{10},
\]
where the last inequality can be quickly verified by evaluating it only at $\delta = 0$ and another end-point, say $\delta = 1$ (using that the left hand side is convex in $\sqrt{\delta}$, whereas the right hand side is linear).

In either case, we have obtained $p_a(0) \leq \frac{1}{\sqrt{2}} - \frac{1}{10}\sqrt{\delta(a)}$.

\subsubsection{Stability away from the maximiser: $\delta(a) > \frac{1}{2000}$.}
Our goal here is to show a constant factor improvement,
\begin{equation}\label{eq:delta-large-goal}
p_a(0) \leq \frac{1-4\cdot 10^{-5}}{\sqrt{2}}.
\end{equation}

As before, we can assume that $u \geq v$. In particular,
\[
v \leq \frac12\left(u+v\right) = \frac12\left(\sqrt{2}-\delta(a)/\sqrt{2}\right) < \frac{1-\frac{1}{4000}}{\sqrt{2}},
\]
and since $a_{n+1} = -v$ is the smallest (negative) weight, we also have for each $j$,
\begin{equation}\label{eq:all-small}
a_j \geq -v \geq -\frac{1-\frac{1}{4000}}{\sqrt{2}}.
\end{equation}
We fix a small positive constant $\eta < \frac12$ (specified soon) and use different arguments depending on the value of $u$.

\emph{Case 1: $u \leq \sqrt{\frac12 + \eta}$.} There is a further dichotomy in our argument: if all weights $a_j$ have magnitudes bounded away from $\frac{1}{\sqrt2}$ (the negative weights do, by \eqref{eq:all-small}), we employ the Fourier-analytic bound from Lemma \ref{lm:Fourier}, otherwise, {\red if there} is a positive weight $a_k$ close to $\frac{1}{\sqrt2}$, we can pair up its summand $a_kX_k$ with $a_1X_1$ and get a good bound from Lemma \ref{lm:density2max}. Specifically, we proceed as follows.

\emph{Case 1.1: ${\red a_2} \leq \frac{1-\frac{1}{4000}}{\sqrt{2}}$.} Let $m = \max_{j>1} |a_j|$. By virtue of \eqref{eq:all-small}, we have $m\leq \frac{1-\frac{1}{4000}}{\sqrt{2}}$. Lemma \ref{lm:Fourier} yields
\begin{align*}
\sqrt{2}p_a(0)  \leq \Psi(u^2)^{u^2}\prod_{j>1} \Psi(a_j^2)^{a_j^2} &\leq \Psi\left(u^2\right)^{u^2}\Psi\left(m^2\right)^{\sum_{j>1} a_j^2} \\
&= \Psi\left(u^2\right)^{u^2}\Psi\left(m^2\right)^{1-u^2} \\
&\leq  \Psi\left(\tfrac12 + \eta\right)^{u^2}\Psi\left(m^2\right)^{1-u^2},
\end{align*}
using also the monotonicity of $\Psi$. {\red Recall that $\Psi(\frac12) = 1$.} Since $\Psi\left(\tfrac12 + \eta\right) > 1$  and $\Psi(m^2) < 1$, the right hand side is increasing in $u^2$, so we get a further upper bound by replacing it with $\frac{1}{2}+\eta$,
\[
\sqrt{2}p_a(0) \leq \Psi\left(\tfrac12 + \eta\right)^{\tfrac12 + \eta}\Psi\left(m^2\right)^{\tfrac12 - \eta} \leq  \Psi\left(\tfrac12 + \eta\right)^{\tfrac12 + \eta}\Psi\left(\tfrac12 - \tfrac{1}{8000}\right)^{\tfrac12 - \eta}.
\]
Using the explicit form of $\Psi$ from Lemma \ref{lm:Fourier}, {\red $\Psi(x) = \frac{1}{\sqrt{2\pi x}}\frac{\Gamma\left(\frac{1}{2x}-\frac{1}{2}\right)}{\Gamma\left(\frac{1}{2x}\right)}$,} it can be checked numerically that the right hand side is less than $1 - 5\cdot 10^{-5}$ for $\eta = 5\cdot 10^{-5}$. This gives \eqref{eq:delta-large-goal}.

\emph{Case 1.2: ${\red a_2} > \frac{1-\frac{1}{4000}}{\sqrt{2}}$.} We split the sum $S = \sum_{j=1}^{n+1} a_jX_j$ as
\[
S = X + Y, \qquad X = a_1X_1 + {\red a_2X_2}, \qquad Y = \sum_{{\red j > 2}} a_jX_j.
\]
Let $f_X$ and $f_Y$ be the densities of $X$ and $Y$, respectively. By H\"older's inequality and Lemma \ref{lm:density2max}, we obtain
\[
p_a(0) = (f_X \star f_Y) (0) \leq \|f_X\|_\infty \|f_Y\|_1 = \|f_X\|_\infty \leq \frac{1}{e\min\{a_1, {\red a_k}\}} = \frac{1}{e{\red a_2}},
\]
where $f_X \star f_Y$ denotes the convolution of $f_X$ and $f_Y$.
The bound ${\red a_2} > \frac{1-\frac{1}{4000}}{\sqrt{2}}$ gives \eqref{eq:delta-large-goal} with a generous margin.

\emph{Case 2: $u > \sqrt{\frac12 + \eta}$.} Readily, by Lemma \ref{lm:a1},
\[
p_a(0) \leq \frac{1}{2u} < \frac{1}{\sqrt{2+4\eta}} = \frac{1}{\sqrt{2}}\frac{1}{\sqrt{1+10^{-4}}} < \frac{1-4\cdot10^{-5}}{\sqrt{2}},
\]
which gives \eqref{eq:delta-large-goal} and finishes the proof.

\section{Proof of Lipschitzness: Theorem \ref{thm:sec-codim1}}

\subsection{Auxiliary results}

We begin by recalling three results. The first one is a well-known fact that sections of isotropic convex bodies are roughly constant (going back to Hensley, \cite{Hen-gen}).

\begin{lemma}\label{thm:Hensley}
Let $1 \leq \ell < n$. Let $K$ be an isotropic convex body in $\R^n$. For every codimension $\ell$ subspace $E$ in $\R^n$, we have
\[
\frac{1}{\sqrt{2\pi e^3}}\frac{1}{L_K} \leq \vol_{n-\ell}(K \cap E)^{1/\ell} \leq \frac{\sC_\ell}{L_K}.
\]
The constant $\sC_\ell$ is defined in \eqref{eq:Cl}.
\end{lemma}
\begin{proof}
To sketch a standard argument, consider $f(x) = \vol_{n-\ell}(K\cap (x + E))$, $x \in E^\perp$, which defines a log-concave probability density on $E^\perp \simeq \R^\ell$ with mean $0$ and covariance matrix $L_K^2I$, since $K$ is isotropic. Then $\tilde f(x) = L_K^\ell f(L_K x)$ is isotropic and the upper bound follows from the definition of $\sC_\ell$, $\vol_{n-\ell}(K \cap E) = f(0) = L_K^{-\ell}\tilde f(0) \leq L_K^{-\ell}\|\tilde f\|_\infty \leq L_K^{-\ell}\sC_\ell^\ell$. 
The lower bound follows from the folklore result that for an arbitrary probability density $g$ on $\R^\ell$ with $\int_{\R^\ell}|x|^2g(x) \dd x = \ell$, we have $\|g\|_\infty \geq \frac{1}{(\ell+2)^{\ell/2}\vol_\ell(B_2^\ell)}$ (by a standard comparison argument with the uniform distribution on a Euclidean ball, see e.g. Lemma 6 in \cite{Ball}). Moreover, for a log-concave density with mean $0$, $\|f\|_\infty \leq e^\ell f(0)$ (see, e.g. Fradelizi's work \cite{Fr}). Combining the two bounds yields the lower bound.
\end{proof}

The second result we shall need is a functional version of Busemann's theorem from \cite{Bus}, essentially due to Ball (who however assumed symmetry).

\begin{theorem}[Ball, \cite{Ball}]\label{thm:Ball}
Let $f\colon \R^\ell \to [0,+\infty)$ be a log-concave function satisfying $0 < \int_{\R^\ell} f < \infty$. Define
\[
N(x) = \frac{1}{\int_0^\infty f(tx) \dd t}, \qquad x \in \R^\ell, \ x \neq 0,
\]
extended with $N(0) = 0$. Then $N$ satisfies, $0 < N(x) < \infty$ and $N(\lambda x) = \lambda N(x)$ for all $x \neq 0$, $\lambda > 0$, as well as $N(x+y) \leq N(x) + N(y)$, for all $x, y \in \R^\ell$.
\end{theorem}

Ball stated this result with the additional assumption on $f$ being \emph{even}, in which case $N$ is in fact a norm, but it is implicit in his proof that the above holds as well (see also Theorem 2.1 and Remark 2.6 in \cite{Bob}, as well as Theorem 3.1 in \cite{CFPP}, {\red or Theorem 12 in \cite{CK}}).

Finally, we shall need an estimate on the proportion of mass of centred log-concave functions retained on half-lines.

\begin{theorem}[Meyer-Nazarov-Ryabogin-Yaskin, \cite{MNRY}]\label{thm:MNRY}
Let $f\colon \R^\ell \to [0,+\infty)$ be a log-concave function satisfying $0 < \int_{\R^\ell} f < \infty$ and $\int_{\R^\ell} xf(x) \dd x = 0$. Then for every unit vector $\theta$ in $\R^\ell$, we have
\[
\int_0^\infty f(t\theta) \dd t \geq e^{-\ell}\int_{-\infty}^\infty f(t\theta) \dd t.
\]
\end{theorem}

The case $\ell = 1$ is the classical Gr\"unbaum lemma originating in \cite{Gru}. Interestingly, constant $e^{-\ell}$ is sharp. For further generalisations and extensions, including sharp estimates for sections of convex bodies, see \cite{MSZ}.

\subsection{Proof of Theorem \ref{thm:sec-codim1}}
We follow closely \cite{ABB} which assumes the symmetry of $K$, so we treat in detail only those parts of the argument thanks to which we can omit the symmetry assumption. 

\emph{Case 1: $F^\perp \cap E = \{0\}$.} There is an orthonormal basis $(u_j)_{j=1}^\ell$  of $E$ such that $v_j = \frac{P_F(u_j)}{|P_F(u_j)|}$, $j \leq \ell$, form an orthonormal basis of $F$ (Lemma 4.1 in \cite{GM}). Then $\scal{u_i}{v_j} = |P_F(u_i)|\delta_{ij}$. The idea is to bound the sections swapping one basis vector at a time, that is we consider the sequence of subspaces
\begin{align*}
E_0 &= \text{span}\{u_1, u_2, u_3, \dots, u_\ell\}, \\
E_1 &= \text{span}\{v_1, u_2, u_3, \dots, u_\ell\}, \\ 
E_2 &= \text{span}\{v_1, v_2, u_3, \dots, u_\ell\}, \\
\ldots& \\
E_\ell &= \text{span}\{v_1, v_2, v_3, \dots, v_\ell\},
\end{align*}
so that $E = E_0$, $F = E_\ell$, and plainly,
\[
\big| \vol_{n-\ell}(K \cap E^\perp) - \vol_{n-\ell}(K \cap F^\perp) \big| \leq \sum_{j=1}^\ell \big| \vol_{n-\ell}(K \cap E_j^\perp) - \vol_{n-\ell}(K \cap E_{j-1}^\perp) \big|.
\]
We fix $j \leq \ell$ and bound the $j$-th term. By the construction of the bases $(u_i)$ and $(v_i)$, the subspaces $E_{j-1}^\perp$ and $E_j^\perp$ are of the form ${\red \bar E^\perp} \oplus a$ and ${\red \bar E^\perp} \oplus b$ with {\red $\bar E$ defined by} ${\red \bar E^\perp} = E_{j-1}^\perp \cap v_j^\perp = E_j^\perp \cap u_j^\perp$ and $a = P_{E_{j-1}^\perp} (v_j)$, $b = -P_{E_j^\perp} (u_j)$. Then, $\lambda = |a|=|b| = \sqrt{1- \scal{u_j}{v_j}^2}$. 
We define functions on ${\red \bar E}$,
\[
f(x) = \vol_{n-\ell-1}\big(K \cap (x + {\red \bar E^\perp})\big), \qquad x \in {\red \bar E}
\]
and
\[
N_+(x) = \frac{1}{\int_0^\infty f(tx) \dd t}, \qquad N_-(x) = \frac{1}{\int_{-\infty}^0 f(tx) \dd t}, \qquad x \in {\red \bar E}, x \neq 0,
\]
both extended by $0$ at $0$. They allow us to avoid assuming that $K$ is symmetric. By Theorem \ref{thm:Ball}, each $N_{\pm}$ satisfies the triangle inequality on ${\red \bar E}$. Introducing the normalised vectors, $\tilde a = \lambda^{-1}a$, $\tilde b = \lambda^{-1}b$, we have,
\begin{align*}
\big| \vol_{n-\ell}(K \cap E_j^\perp) - \vol_{n-\ell}(K \cap E_{j-1}^\perp) \big| &= \left|\frac{1}{N_+(\tilde a)} + \frac{1}{N_-(\tilde a)} - \frac{1}{N_+(\tilde b)} - \frac{1}{N_-(\tilde b)}\right| \\
&\leq \left| \frac{1}{N_+(\tilde a)} - \frac{1}{N_+(\tilde b)} \right| + \left| \frac{1}{N_-(\tilde a)} - \frac{1}{N_-(\tilde b)} \right|.
\end{align*}
Moreover,
\begin{align*}
|N_{\pm}(\tilde a)^{-1} - N_{\pm}(\tilde b)^{-1}| = \left|\frac{N_{\pm}(\tilde a) - N_{\pm}(\tilde b)}{N_{\pm}(\tilde a)N_{\pm}(\tilde b)}\right|.
\end{align*}
By the triangle inequality,
\[
-N_\pm(\tilde b- \tilde a) \leq N_\pm(\tilde a) - N_\pm(\tilde b) \leq N_\pm(\tilde a- \tilde b).
\]
As a result,
\[
|N_{\pm}(\tilde a)^{-1} - N_{\pm}(\tilde b)^{-1}| \leq \frac{\max\left\{N_\pm\left(\frac{\tilde a - \tilde b}{|\tilde a-\tilde b|}\right), N_\pm\left(\frac{\tilde b - \tilde a}{|\tilde a - \tilde b|}\right)\right\}}{N_{\pm}(\tilde a)N_{\pm}(\tilde b)}\lambda^{-1}|a- b|,
\]
using homogeneity and $|\tilde a-\tilde b| = \lambda^{-1}|a- b|$. 
{\red Also since, \[
|a-b| = \sqrt{1- \scal{u_j}{v_j}^2}|u_j-v_j|,
\]} 
we obtain
\[
\lambda^{-1}|a- b| = |u_j-v_j|.
\]
It remains to upper and lower bound $N_\pm(\theta)$
for an arbitrary unit vector $\theta$ in ${\red \bar E}$. We shall do it for $N_+$ and of course the argument for $N_-$ is identical (by changing $f(\cdot)$ to $f(-\cdot)$). By Theorem \ref{thm:MNRY},
\[
\frac{1}{\int_{-\infty}^\infty f(t\theta) \dd t} \leq 
 N_+(\theta) \leq \frac{e^{\ell+1}}{\int_{-\infty}^\infty f(t\theta) \dd t}.
\]
Moreover, $\int_{-\infty}^\infty f(t\theta) \dd t = \vol_{n-\ell}(K \cap ({\red \bar E^\perp} \oplus \theta))$, so Lemma \ref{thm:Hensley} yields
\[
\sC_\ell^{-\ell}L_K^\ell \leq N_+(\theta) \leq e^{\ell+1}(2\pi e^3)^{\ell/2}L_K^\ell.
\]
Putting everything together, we obtain
\begin{align*}
|N_{\pm}(\tilde a)^{-1} - N_{\pm}(\tilde b)^{-1}| \leq (2\pi)^{\ell/2}e^{5\ell/2+1}\frac{\sC_\ell^{2\ell}}{L_K^\ell}|u_j-v_j|.
\end{align*}
Hence, by the Cauchy-Schwarz inequality and a basic estimate $\sqrt{\sum_{j=1}^\ell |u_j-v_j|^2} \leq {\red \sqrt{2\ell}}\cdot d(E,F)$ (see Proposition 2.2 in \cite{ABB}),
\begin{align*}
\big| \vol_{n-\ell}(K \cap E^\perp) - \vol_{n-\ell}(K \cap F^\perp) \big| &\leq \sqrt{2\ell}(2\pi)^{\ell/2}e^{5\ell/2+1}\frac{\sC_\ell^{2\ell}}{L_K^\ell}d(E,F) \\
&\leq  e^{5\ell}\frac{\sC_\ell^{2\ell}}{L_K^\ell}d(E,F).
\end{align*}

\emph{Case 2: $E' = F^\perp \cap E \neq \{0\}$.} Then we have $E = E' \oplus E''$ with $F^\perp \cap E'' = \{0\}$. We choose an orthonormal basis of $E''$ as in Case 1, say $(u_i)_{i=1}^k$, which gives an orthonormal system $v_i = P_F(u_i)/|P_F(u_i)|$ in $F$. We complement it to an orthonormal basis $(v_i)_{i=1}^\ell$ of $F$. We also complement  $(u_i)_{i=1}^k$ to an orthonormal basis $(u_i)_{i=1}^\ell$ of $E$. Note that $\text{span}\{u_i,\, k < i \leq \ell\}$ and $\text{span}\{v_i,\, k < i \leq \ell\}$ are orthogonal. Swapping the $u_i$ by $v_i$, we then define the chain of subspaces $E = E_0, E_1, \dots, E_\ell = F$ as in Case~1 and proceed exactly as therein to bound $\big| \vol_{n-\ell}(K \cap E_j^\perp) - \vol_{n-\ell}(K \cap E_{j-1}^\perp) \big|$. Bounding the terms $j \leq k$ is unchanged, whereas for the terms $j > k$, calculations are even simpler because then $u_j$ and $v_j$ are orthogonal.

\textbf{Acknowledgments.}
We should very much like to thank Giorgos Chasapis and Grigoris Paouris for their comments. We  appreciate the hospitality and excellent working conditions at the Hausdorff Research Institute for Mathematics in Bonn during the programme ``Synergies between modern probability, geometric analysis and stochastic geometry".

\end{document}